\documentclass{scrartcl}
\usepackage[utf8]{inputenc}
\usepackage{graphicx}
\usepackage{amsfonts,amsmath,amsthm,mathtools}
\usepackage{units}
\usepackage{bbm}

\DeclareMathOperator{\Var}{Var}
\newtheorem{theorem}{Theorem}

\newtheorem{remark}{Remark}

\newtheorem{proposition}{Proposition}
\newcommand{\myp}{f}
\title{Convergence of Optimal Expected Utility for a Sequence of 
Binomial Models}
\subtitle{Dedicated to the memory of Mark Davis}
\author{Friedrich Hubalek \and Walter Schachermayer\thanks{Walter Schachermayer acknowledges support from the Austrian Science Fund (FWF) 
under grant~P28861 and by  the Vienna Science and Technology Fund (WWTF) through project MA16--021.}}
\date{September 21, 2020}
\begin{document}
\maketitle
\begin{abstract} 
We analyze the convergence of expected utility under the approximation 
of the Black-Scholes model by binomial models. In a recent paper by 
D.~Kreps and W.~Schachermayer a surprising and somewhat 
counter-intuitive example was given: such a convergence may, in 
general, fail to hold true. This counterexample is based on a binomial 
model where the i.i.d.\ logarithmic one-step increments have strictly 
positive third moments. This is the case, when the up-tick of the 
log-price is larger than the down-tick.

In the paper by D.~Kreps and W.~Schachermayer it was left as an open 
question how things behave in the case when the down-tick is larger 
than the up-tick and --- most importantly --- in the case of the 
symmetric binomial model where the up-tick equals the down-tick. Is 
there a general positive result of convergence of expected utility in 
this setting?

In the present note we provide a positive answer to this question. It 
is based on some rather fine estimates of the convergence arising in 
the Central Limit Theorem.
\end{abstract}
\section{Introduction and Main Result\label{sec-intro}}
We adopt the setting of the paper \cite{W175} which in turn is based on 
David Kreps' monograph \cite{Kreps}. We assume that the reader is 
familiar with \cite{W175} and use the notation and definitions of this 
paper.

In \cite[Section~9]{W175} a counterexample is presented which shows 
that for the positively skewed, asymmetric binomial model the following 
fact holds true: the approximation of the Black-Scholes model by this 
binomial model does not yield convergence of the corresponding 
portfolio optimization problems if one does not impose extra conditions 
on the utility function such as the condition of {\em reasonable 
asymptotic elasticity}. In the present note we show the reassuring fact 
that for the {\em symmetric} binomial model we have a positive result 
without imposing extra conditions on the utility function $U(\cdot)$.

We consider a general utility function $U:\mathbb R_+\to\mathbb R$ 
satisfying Assumption~(3.1) of~\cite{W175}, i.e., being strictly 
increasing, strictly concave, smooth, and satisfying the Inada 
conditions $\lim_{x\to0}U'(x)=\infty$ and $\lim_{x\to\infty}U'(x)=0$. 
As in \cite{W175} and \cite{W90} we denote by $V:\mathbb R_+\to\mathbb 
R$ the conjugate function of $U$, i.e., 
$V(y)=\sup_{x>0}\left\{U(x)-xy\right\}$.

For arbitrary $p\in(0,1)$ we consider and i.i.d.\ sequence 
$(\alpha_n)_{n=1}^\infty$ of Bernoulli variables with
\begin{equation}
P[\alpha_n=0]=1-p,\qquad P[\alpha_n=1]=p,
\end{equation} 
where $p\in(0,1)$. Denote by $\zeta_n$ the corresponding standardized 
variables
\begin{equation}\label{zetanp}
\zeta_n=\frac{\alpha_n-p}{\sqrt{p(1-p)}},
\end{equation}
so that $E[\zeta_n]=0$ and $\Var[\zeta_n]=1$. Denote by~$\xi_{n,k}$ the 
scaled partial sums
\begin{equation}
\xi_{n,k}=\frac1{\sqrt n}\sum_{j=1}^k\zeta_j
\end{equation}
and set
\begin{equation}
z_{n,k}=\frac{k-np}{\sqrt{np(1-p)}},\qquad f_{n,k}=\binom{n}{k}p^k(1-p)^{n-k},
\qquad k=0,\ldots,n.
\end{equation}
Then $f_{n,k}=P[\xi_{n,n}=z_{n,k}]$.

If we set $S^{(n)}(k/n)=e^{\xi_{n,k}}$ and extend $S^{(n)}$ by 
interpolation to continuous-time processes as in \cite{W175}, then 
$(S^{(n)}(t))_{0\leq t\leq1}$ approximates the Black-Scholes-Merton 
(BSM) model. The latter is defined by
\begin{equation}
S(t)=\exp\left(\omega(t)\right),\qquad 0\leq t\leq 1,
\end{equation}
where $(\omega(t))_{0\leq t\leq1}$ is a standard Brownian motion.

In fact, for the present purposes, we only need the random variable 
$\omega(1)$ under the standard normal measure $\mathcal P$ as well as 
under the normalized binomial measure, i.e., the law of $\xi_{n,n}$, 
denoted by $\mathcal P_n$.

To wit, the distribution of the random variable $\omega(1)$ under 
$\mathcal P_n$ equals the distribution of 
$\xi_{n,n}=\frac{\zeta_1+\ldots+\zeta_n}{n^{\frac12}}$, where 
$(\zeta_i)_{i=1}^n$ are independent copies of $\zeta$ and, as 
in~\cite{W175}, we denote by $u(x)$, respectively by $u_n(x)$, the 
supremal expected utility which an agent can achieve in the BSM 
economy, respectively in the $n$-th discrete-time economy, if her 
initial wealth is~$x$.

As in \cite{W175} we define the value functions $u(x)$ and $u_n(x)$ as
\begin{equation}
u(x)=\sup\left\{\mathbb E_{\mathcal P}\left[U(X)\right]:\mathbb E_{\mathcal P^*}[X]\leq x\right\}
\end{equation}
and
\begin{equation}
u_n(x)=\sup\left\{\mathbb E_{\mathcal P_n}\left[U(X_n)\right]:\mathbb E_{\mathcal P^*_n}[X_n]\leq x\right\},
\end{equation}
where $\mathcal P^*$ and $\mathcal P^*_n$ are the unique equivalent martingale measures
pertaining to the Black-Scholes model and its $n$-th approximation, respectively.

When $p\in(0,1/2)$ we have $E[\zeta_n^3]>0$ and things go astray, as 
demonstrated by the counterexample in Section~9 of \cite{W175}. 
Therefore we focus on the other case, when $p\in[1/2,1)$. 

Here is our main result.
\begin{theorem}\label{x.x}
If $U$ and $\zeta$ are as above with $p\in[1/2,1)$, we have
\begin{equation}\label{A3}
u(x)=\lim_{n\to\infty}u_n(x),\qquad x>0.
\end{equation}
\end{theorem}
The theorem will follow from the subsequent more technical version 
of~(\ref{A3}).
\begin{proposition}\label{prop-limsup}
Denote by $v(y)$, respectively $v_n(y)$, the value functions conjugate 
to $u(x)$, respectively $u_n(x)$, as in \cite{W175}. We then have
\begin{equation}\label{A3a}
\limsup_{n\to\infty}v_n(y)\leq v(y),\qquad y>0.
\end{equation}
\end{proposition}
\begin{proof}[Proof of Theorem~\ref{x.x} 
admitting Proposition~\ref{prop-limsup}]
We deduce from \cite[Proposition~2]{W175} and standard results on 
conjugate functions that the reverse inequality to~(\ref{A3a}) does 
hold true, i.e.,
\begin{equation}\label{A4}
\liminf_{n\to\infty}v_n(y)\geq v(y), \qquad y>0.
\end{equation}
Admitting Proposition~\ref{prop-limsup}, formulas~(\ref{A3a}) 
and~(\ref{A4}) imply the equality
\begin{equation}\label{A4a}
\lim_{n\to\infty}v_n(y)=v(y),\qquad y>0.
\end{equation}
Using once again standard results on conjugate functions (compare 
\cite{W175}) we obtain~(\ref{A3}) from~(\ref{A4a}).
\end{proof}
A key ingredient for the proof of Proposition~\ref{prop-limsup} are 
estimates for the tails of the standardized binomial distributions in 
terms of the standard Gaussian tails. 

Let $\xi$ be a standard normal random variable and denote its density 
by $\phi(x)=\frac1{\sqrt{2\pi}}e^{-\frac{x^2}{2}}$. Set
\begin{equation}
F_n(x)=P[\xi_{n,n}\leq x],\qquad
\Phi(x)=P[\xi_\leq x]
\end{equation}
and
\begin{equation}
\bar F_n(x)=1-F_n(x)=P[\xi_{n,n}>x],\qquad
\bar\Phi(x)=1-\Phi(x)=P[\xi>x].
\end{equation}
\begin{proposition}\label{prop-bintail}
Suppose $p\in[1/2,1)$, then there is $C>0$ such that, for $n\geq1$, we 
have
\begin{equation}\label{Eqn-5a}
\myp_{n,k}\leq C\cdot\frac2{\sqrt{n}}\phi(z_{n,k-1}),\qquad 0\leq k\leq \left\lceil np\right\rceil,
\end{equation}
\begin{equation}\label{Eqn-5}
\myp_{n,k}\leq C\cdot\frac2{\sqrt{n}}\phi(z_{n,k+1}),\qquad \left\lfloor np\right\rfloor\leq k\leq n.
\end{equation}
Furthermore
\begin{equation}\label{Eqn-left-global}
F_n(x)\leq C\Phi(x),\qquad x\leq0
\end{equation}
and
\begin{equation}\label{Eqn-global}
\bar F_n(x)\leq C\bar\Phi(x),\qquad x\geq0.
\end{equation}
\end{proposition}
\begin{remark}\label{Rem-area}
The terms $\frac2{\sqrt{n}}\phi(z_{n,k\pm1})$ in~(\ref{Eqn-5a}) 
and~(\ref{Eqn-5}) are a lower bound for the area under the density 
$\varphi(x)$ between $z_{n,k-1}$ and $z_{n,k}$ on the left and 
$z_{n,k}$ and $z_{n,k+1}$ on the right tail, respectively.
\end{remark}
We prove Proposition~\ref{prop-limsup} and 
Proposition~\ref{prop-bintail} first for the symmetric case $p=1/2$ in 
Section~\ref{sec-sym}, as this case allows for several simplifications 
and the main arguments are more transparent. We then provide the 
slightly more technical details for the asymmetric case $p\in(1/2,1)$ 
in Section~\ref{sec-asym}.
\section{The symmetric case\label{sec-sym}}
\begin{proof}[Proof of Proposition~\ref{prop-limsup} (symmetric case) 
admitting Proposition~\ref{prop-bintail}]
As in \cite[Section~3]{W175} we write
\begin{equation}\label{v}
v(y)=\mathbb E_{\mathcal P}[V(yZ)],\qquad y>0,
\end{equation}
and
\begin{equation}\label{vn}
v_n(y)=\mathbb E_{\mathcal P_n}[V(yZ_n)],\qquad y>0,
\end{equation}
where $V(y)=\sup\{U(x)-xy:x>0\}$ is the conjugate function of $U$.

The random variables $Z$ and $Z_n$ are the densities of the (unique) 
equivalent martingale measures $\mathcal P^*$ and $\mathcal P^*_n$ with 
respect to $\mathcal P$ and $\mathcal P_n$, respectively, i.e., 
$Z=\frac{d\mathcal P^*}{d\mathcal P}$ 
and $Z_n=\frac{d\mathcal P^*_n}{d\mathcal P_n}$. They are of the form
\begin{equation}
Z=\exp\left(-\frac{\omega(1)}{2}-\frac18\right)
\end{equation}
and
\begin{equation}\label{nineteen}
Z_n=\exp\left(-a_n\omega(1)-b_n\right).
\end{equation}
In the symmetric case the calculations from  \cite[Section~6]{W175} 
simplify, and we have that
\begin{equation}\label{A5}
a_n=\frac12,\qquad
b_n=n\log\cosh\left(\frac1{2\sqrt n}\right).
\end{equation}
It follows that $b_n$ increases to $1/8$ as $n\to\infty$.

Fix $y>0$ such that $v(y)<\infty$, otherwise (\ref{A3a}) is certainly true.
Denote by $H:\mathbb R\to\mathbb R$ the function
\begin{equation}\label{FunF}
H(x)=V\left(y\exp\left(-\frac{x}{2}-\frac18\right)\right),\qquad x\in\mathbb R,
\end{equation}
and by $H_n:\mathbb R\to\mathbb R$ the function
\begin{equation}\label{FunFn}
H_n(x)=V\left(y\exp\left(-\frac{x}{2}-b_n\right)\right),\qquad x\in\mathbb R.
\end{equation}
Clearly, these functions are increasing on $\mathbb R$. Note, however, that
they are not necessarily concave.
We know that
\begin{align}\label{v1finite}
v(y)&=\mathbb E_{\mathcal P}[H(\omega(1))]=\int_{\mathbb R}H(x)\phi(x)dx<\infty
\shortintertext{while}
v_n(y)&=\mathbb E_{\mathcal P_n}[H_n(\omega(1))]=\sum_{k=0}^nH_n(z_{n,k})\myp_{n,k},\label{Hvn}
\end{align}
where $\phi(x)$ is the standard normal density and $\myp_{n,k}$ are the 
binomial probabilities as in Section~\ref{sec-intro}.

As $H_n(x)\leq H(x)$ for all $x\in\mathbb R$, in order to 
show~(\ref{A3a}), it will suffice to show
\begin{equation}\label{Hvn-bis}
\limsup_{n\to\infty}\mathbb E_{\mathcal P_n}[H(\omega(1))]\leq\mathbb E_{\mathcal P}[H(\omega(1))].
\end{equation}

In order to show~(\ref{Hvn-bis}) the crucial estimate is the uniform 
integrability of the random variables $H(\omega(1))$ under $\mathcal 
P_n$. More precisely, we need the following estimates~(\ref{A8a}) 
and~(\ref{A8}). For $\varepsilon>0$ there is $M>0$ such that
\begin{equation}\label{A8a}
\mathbb E_{\mathcal P_n}\left[H(\omega(1))\mathbbm{1} _{\{H(\omega(1))>M\}}\right]=
\sum_{k=0}^nH(z_{n,k})\mathbbm{1} _{\{H(z_{n,k})>M\}}\myp_{n,k}<\varepsilon,
\end{equation}
and
\begin{equation}\label{A8}
\mathbb E_{\mathcal P_n}\left[|H(\omega(1))|\mathbbm{1} _{\{H(\omega(1))<-M\}}\right]=
\sum_{k=0}^n|H(z_{n,k})|\mathbbm{1} _{\{H(z_{n,k})<-M\}}\myp_{n,k}<\varepsilon,
\end{equation}
uniformly in $n\in\mathbb N$. Formulas~(\ref{A8a}) and~(\ref{A8}) 
correspond to the formulas \cite[(8.12) and (8.14)]{W175}. 

First we consider $M>H(1)_+$. If $H(z_{n,k})>M$, then $n/2<k\leq n$ and 
we are in a position to invoke formula~(\ref{Eqn-5}) of 
Theorem~\ref{prop-bintail}, which gives an estimate on the right tail 
of the binomial distribution as compared to the normal one. More 
precisely, there is a universal constant $C>0$ such that, for every 
$n\in\mathbb N$ and $n/2<k\leq n$ and all $x\in (z_{n,k},z_{n,k+1})$
\begin{equation}
f_{n,k}\leq C\cdot\frac2{\sqrt n}\phi(x),\qquad
0\leq H(z_{n,k})\leq H(x).
\end{equation}
Thus
\begin{equation}\label{Hright}
\sum_{k=0}^nH(z_{n,k})I_{\{H(z_{n,k})>M\}}f_{n,k}\leq
C\int_{-\infty}^{+\infty}H(x)I_{\{H(x)>M\}}\phi(x)dx.
\end{equation}
It follows from (\ref{v1finite}) that the right-hand side of 
(\ref{Hright}) can be made smaller than $\varepsilon$ for sufficiently 
large $M$.

A similar estimate applies for the left tail. We consider now $M>H(-1)_-$.
If $H(z_{n,k})<-M$ then $0\leq k<n/2$ and
we invoke formula~(\ref{Eqn-5a}) of Theorem~\ref{prop-bintail}.
We get now for every $n\in\mathbb N$ and $0\leq k<n/2$ and 
all $x\in (z_{n,k-1},z_{n,k})$
\begin{equation}
f_{n,k}\leq C\cdot\frac2{\sqrt n}\phi(x),\qquad
H(x)\leq H(z_{n,k})\leq0.
\end{equation}
Thus
\begin{equation}\label{Hleft}
\sum_{k=0}^n|H(z_{n,k})|I_{\{H(z_{n,k})<-M\}}f_{n,k}\leq
C\int_{-\infty}^{+\infty}|H(x)|I_{\{H(x)<-M\}}\phi(x)dx.
\end{equation}
Again, it follows from (\ref{v1finite}) that the right-hand side of 
(\ref{Hleft}) can be made smaller than $\varepsilon$ for sufficiently 
large $M$.

Using the well-known weak convergence of $\mathcal P_n$ to $\mathcal P$ 
and the uniform integrability conditions we can deduce (\ref{Hvn}), see 
\cite[Thm~2.20, p.17]{vdV1998}.

Finally we consider the case $y=y_0$, where 
$y_0=\inf\{y>0:v(y)<\infty\}$, for the case $y_0>0$. Either 
$v(y_0)=\infty$ in which case~(\ref{A3a}) holds trivially. Or 
$v(y_0)<\infty$, in which case $v$ is right continuous at $y_0$, see 
\cite{W90}; we therefore may repeat the above argument with $y=y_0$.

This finishes the proof of Proposition~\ref{prop-limsup}.
\end{proof} 
\begin{proof}[Proof of Proposition~\ref{prop-bintail} (symmetric case)]
Let us start with~(\ref{Eqn-5}). It is enough to prove that there is 
$n_0>0$ such that~(\ref{Eqn-5}) holds for $n\geq n_0$, since $\phi$ is 
strictly positive and there are only finitely many remaining cases that 
can be incorporated in the value of the constant $C$.

Let us consider first the extreme case $k=n$. In this 
case~(\ref{Eqn-5}) follows since $p_{n,n}=2^{-n}$, which decays faster 
than $n^{-1/2}\phi(z_{n,n+1})\approx e^{-\frac{n}{2}}$ for 
$n\to\infty$, as $\log2>1/2$. Here we used the convention 
$z_{n,n+1}=\sqrt{n}+2/\sqrt{n}$, although $z_{n,n+1}$ is not a possible 
value for~$Y_n$. Passing to the other extreme case, we deduce from the 
central limit theorem, that for $k=\lfloor n/2\rfloor$ we have 
$\frac{\sqrt{n}\myp_{n,k}}{2\phi(z_{n,k+1})}\to1$ as $n\to\infty$.

For the remaining cases, i.e., $\lfloor n/2\rfloor<k\leq n-1$, we take 
logarithms and show that
\begin{equation}\label{Eqn-logq}
\log\left(\frac{\sqrt{n}\myp_{n,k}}{2\phi(z_{n,k+1})}\right)
\end{equation}
is bounded from above. To estimate the numerator in~(\ref{Eqn-logq}) we 
use a fine version of Stirling's Formula as given 
in \cite[6.1.38, p.257]{AS}, namely
\begin{equation}\label{Eqn-fine1}
x!=\sqrt{2\pi}x^{x+\frac12}\exp\left(-x+\frac{\theta(x)}{12x}\right),\quad
x>0,
\end{equation}
with  $0<\theta(x)<1$ for all $x>0$. We also note that 
$\lim_{x\to\infty} \theta(x)=0$. We obtain the estimates
\begin{eqnarray}
\log(n!) &\leq & \log\sqrt{2\pi}+\left(n+\frac12\right)\log n -n+\frac1{12n}\label{Eqn-logn}\\
\log(k!) &\geq &\log\sqrt{2\pi}+\left(k+\frac12\right)\log k -k\label{Eqn-logk}\\
\log((n-k)!)&\geq&\log\sqrt{2\pi}+\left(n-k+\frac12\right)\log(n-k)-(n-k).\label{Eqn-lognk}
\end{eqnarray}
This yields an upper bound for the numerator of~(\ref{Eqn-logq}), as
\begin{equation}\label{Eqn-logp}
\log \myp_{n,k}=\log(n!)-\log(k!)-\log((n-k)!)-n\log 2.
\end{equation}
As regards the denominator of~(\ref{Eqn-logq}), we have
\begin{equation}\label{logphi}
\log\phi(z_{n,k+1})=2+2k-\frac{2}{n}-4\frac{k}{n}-2\frac{k^2}{n}-\frac{n}{2}-\log\sqrt{2\pi}.
\end{equation}
Writing $\log k=\log n+\log(k/n)$ and $\log(n-k)=\log n+\log(1-k/n)$ and
combining~(\ref{Eqn-logp}),~(\ref{Eqn-logn}),~(\ref{Eqn-logk}),~(\ref{Eqn-lognk}), and~(\ref{logphi}) yields
\begin{equation}\label{key}
\log\left(\frac{\sqrt{n}\myp_{n,k}}{2\phi(z_{n,k+1})}\right)
\leq g_n\left(\frac{k}{n}\right),
\end{equation}
with $g_n(w)=\alpha(w)n+\beta_n(w)$, where $w\in[\frac12,1]$, and
\begin{equation}
\alpha(w)=-w\log w-(1-w)\log(1-w)-2w(1-w)+\frac12-\log2
\end{equation}
and
\begin{equation}\label{bn}
\beta_n(w)=-\frac12\log w-\frac12\log(1-w)+4w-2+\frac{25}{12n}.
\end{equation}
It remains to show that $g_n(w)$ is bounded from above uniformly in 
$n\in\mathbb N$ and $w\in[1/2,1-1/n]$.

We have $\alpha(\frac12)=\alpha'(\frac12)=\alpha''(\frac12)= 
\alpha'''(\frac12)=0$ and for the forth derivative we have 
$\alpha^{iv}(w)=-2/(1-w)^3-2/w^3<0$ for $w\in(\frac12,1)$, and thus 
each of the functions $\alpha'''(w)$, $\alpha''(w)$, $\alpha'(w)$, and 
$\alpha(w)$ is strictly negative and decreasing for $w\in(\frac12,1)$.

We have $\beta_n(1/2)=25/(12n)$ and 
$\beta_n'(w)=4-1/(2w)+1/(2(1-w))>0$ for $w\in(\frac12,1)$,
thus $\beta_n(w)$ is strictly positive and strictly increasing 
for $w\in(\frac12,1)$.

For $w\in[1/2,3/4]$ we have $g_n(w)\leq b_n(3/4)$. As
$\lim_{n\to\infty}\beta_n(3/4)=1+\log(2/\sqrt3)<\infty$
it follows that $g_n(w)$ is bounded from above for the
interval under consideration.

For $w\in[3/4,1-1/n]$ we have $g_n(w)\leq\alpha(3/4)n+\beta_n(1-1/n)$. 
Now $\alpha(3/4)<0$ and $\beta_n(1-1/n)\sim\frac12\log n$ as $n\to\infty$. 
Here the second term on the right hand side of~(\ref{bn}) 
is the leading term. Finally we use the fact that $\alpha(3/4)n$ grows 
quicker than $\frac12\log n$ to conclude that $g_n(w)$ is negative for 
$w\in[3/4,1-1/n]$ and sufficiently large $n$.

The proof of~(\ref{Eqn-5a}) is completely symmetric with $g_n(w)$ 
for $w\in[1/2,1-1/n]$ replaced by $g_n(1-w)$ for $w\in[1/n,1/2]$.

Having proved~(\ref{Eqn-5a}) and~(\ref{Eqn-5}) we mentioned already in 
Remark~\ref{Rem-area} how these two inequalties imply 
~(\ref{Eqn-left-global}) and~(\ref{Eqn-global}).
\end{proof}
For the above proof of Proposition~\ref{prop-limsup} the estimates~(\ref{Eqn-5a}) and~(\ref{Eqn-5}) 
involving an unspecified constant $C>0$ is sufficiently strong. But we 
we can do better than that. We may adapt the above argument to yield a constant $C=1+\varepsilon$
for $n$ sufficiently large. 
Indeed, analyzing the above proof of Proposition~\ref{prop-bintail}, we see that the 
above argument also works when we split the interval $(\frac12,1)$ not 
at $w=3/4$, but at a point $\theta\in(1/2,1)$, which is close to $1/2$ 
to obtain a better constant $C$, for large enough $n$. The detailed 
argument is given in the proof of the following proposition, which 
sharpens Proposition~\ref{prop-bintail}.
\begin{proposition}\label{Pro-C}
For any $C>1$ there is $n_0(C)>0$ such that 
equations~(\ref{Eqn-5a})--(\ref{Eqn-5}) 
and~(\ref{Eqn-left-global})--(\ref{Eqn-global}) hold 
for $n\geq n_0(C)$. 
\end{proposition}
\begin{proof}
We consider $\vartheta\in(\frac12,1)$ and proceed as in the proof of 
the Proposition~\ref{prop-bintail} above. We distinguish two cases, $w\in[\frac12,\vartheta]$
and $w\in[\vartheta,1-1/n]$. 
In the first case, when $w\in[1/2,\vartheta]$, we have $g_n(w)\leq \beta_n(\vartheta)$
and
\begin{equation}
\lim_{n\to\infty}\beta_n(\vartheta)=-\frac12\log\vartheta-\frac12\log(1-\vartheta)+4\vartheta-2-\log2.
\end{equation}
The right hand side is increasing in $\vartheta$ and equals zero when $\vartheta=1/2$.
In the second case, for $w\in[\vartheta,1-1/n]$
 we have $g_n(w)\leq \alpha(\vartheta)n+\beta_n(1-1/n)$.
Again $\alpha(\vartheta)<0$ and $\beta_n(1-1/n)\sim\frac12\log n$ as $n\to\infty$,
so that $g_n(w)$ is negative for $w\in[\vartheta,1-1/n]$ and sufficiently large $n$.
\end{proof}
%
\section{The asymmetric case\label{sec-asym}}
\begin{proof}[Proof of Proposition~\ref{prop-limsup} (asymmetric case) 
admitting Proposition~\ref{prop-limsup}]
We fix $p\in(\frac12,1)$ and follow the steps from the symmetric case, but now we get
instead of (\ref{A5}) the following coefficients in~(\ref{nineteen}):
\begin{equation}
a_n=\frac{\sqrt n}{z_{1,1}-z_{1,0}}\log\left(
\frac{p}{1-p}\frac{e^{z_{1,1}/\sqrt n}-1}{1-e^{z_{1,0}/\sqrt n}}
\right),
\end{equation}
and
\begin{equation}
b_n=n\log\left((1-p)e^{-z_{1,0}a_n/\sqrt n}+pe^{-z_{1,1}a_n/\sqrt n}\right).
\end{equation}
with $z_{1,0}=-\sqrt{p/(1-p)}$ and $z_{1,1}=\sqrt{(1-p)/p}$.
Straightforward asymptotic expansions for $n\to\infty$ yield
\begin{equation}
a_n=\frac12-\frac{2p-1}{24\sqrt{p(1-p)}}n^{-\frac12}+\mathcal O(n^{-1}),
\end{equation}
which slightly extends the result that is given in \cite[Sec.6]{W175},
and
\begin{equation}
b_n=\frac18-\frac{1-p+p^2}{576p(1-p)}n^{-1}+\mathcal O(n^{-2}).
\end{equation}
Now we fix an arbitrary $\delta>0$. For $p\in(1/2,1)$ it follows from the
asymptotics that
\begin{equation}\label{ab-bounds}
0<a_n\leq\frac12,\qquad
\frac18-\delta\leq b_n\leq \frac18 
\end{equation}
for all $n$ sufficiently large.
In fact, these inequalities are also true for $p=1/2$
as can be seen from (\ref{A5}).

Instead of (\ref{FunFn}) we now consider
\begin{equation}
H_n(x)=V\left(y\exp\left(-a_nx-b_n\right)\right),\qquad x\in\mathbb R.
\end{equation}
If we are in the right tail and $z_{n,k}\geq0$ then (\ref{ab-bounds})
yields $H_n(z_{n,k})\leq H(z_{n,k})$ and the uniform inegrability follows 
just as in the symmetric case.

If we are in the left tail and $z_{n,k}\leq0$ then (\ref{ab-bounds})
yields $H_n(z_{n,k})\geq\tilde H(z_{n,k})$, where
$\tilde H(x)=V(\tilde ye^{-x/2-1/8})$ with $\tilde y=ye^{\delta}$. Due to
the convexity of $V$ we have $v(\tilde y)>-\infty$
and the uniform inegrability follows 
just as in the symmetric case.
\end{proof}
\begin{proof}[Proof of Proposition~\ref{prop-bintail} (asymmetric case)]
Following the steps from the symmetric case we now get
\begin{equation}\label{logp}
\log f_{n,k}=\log(n!)-\log(k!)-\log((n-k)!)+k\log p+(n-k)\log(1-p).
\end{equation}
and
\begin{equation}\label{logphip}
\log\phi(z_{n,k+1})=\frac{1}{2} (-\log (2)-\log (\pi ))-\frac{(k-n p+1)^2}{2 n (1-p) p}.
\end{equation}
the key inequality~(\ref{key}) becomes
\begin{equation}
\log\left(\frac{\sqrt{np(1-p)}f_{n,k}}{\phi(z_{n,k+1})}\right)
\leq g_n\left(\frac{k}{n}\right),
\end{equation}
with $g_n(w)=\alpha(w)n+\beta_n(w)$, where $w\in[p,1]$, and
\begin{multline}
\alpha(w)=
-w\log w-(1-w)\log(1-w)\\
+\frac{w^2}{2p(1-p)}
+\left(\log\frac{p}{1-p}-\frac1{1-p}\right)w
+\frac{p}{2(1-p)}+\log(1-p)
\end{multline}
and
\begin{equation}
\beta_n(w)=-\frac12\log\left(w(1-w)\right)
+\frac{w}{p(1-p)}-\log2
-\frac1{1-p}+\left(\frac1{12}+\frac1{2p(1-p)}\right)\frac1n.
\end{equation}
Again we have $\alpha(p)=\alpha'(p)=\alpha''(p)=0$. As regards
the third derivative we find $\alpha'''(w)=\frac{1-2 p}{(p-1)^2 p^2}<0$ for $w\in(p,1)$,
and thus $\alpha'''(w)$, $\alpha''(w)$, $\alpha'(w)$, and $\alpha(w)$ again are strictly negative and 
decreasing for $w\in(p,1)$.

We have $\beta_n(p)=\frac{1}{12} \left(\frac{6}{n p-n p^2}+\frac{1}{n}-6 \log (-4 (p-1) p)\right)$ 
and $\beta_n'(w)=\frac{1}{p-p^2}+\frac{1}{2-2 w}-\frac{1}{2 w}>0$ for $w\in(p,1)$,
thus $\beta_n(w)$ is strictly positive and strictly increasing for $w\in(p,1)$.

Similarly as in the proof of Proposition~\ref{Pro-C}, fix $\vartheta\in(p,1)$. 
For $w\in[p,\vartheta]$ we have $g_n(w)\leq\beta_n(\vartheta)$. As
$\lim_{n\to\infty}\beta_n(\vartheta)=-1/2\log(\vartheta(1-\vartheta))+\frac{h}{p(1-p)}-\frac1{1-p}-\log2$
it follows that $g_n(w)$ is bounded from above for the
interval under consideration.
For $w\in[\vartheta,1-1/n]$ we have $g_n(w)\leq\alpha(\vartheta)n+\beta_n(1-1/n)$.
Now $\alpha(\vartheta)<0$ and $\beta_n(1-1/n)\sim1/2\log n$ as $n\to\infty$
and thus $g_n(w)$ is negative for $w\in[\vartheta,1-1/n]$ and sufficiently large $n$.
\end{proof}
\bibliography{binut}
\end{document}